\documentclass[reqno]{amsart}

\usepackage{amssymb}
\usepackage{graphicx}
\usepackage{amscd}
\usepackage[hidelinks]{hyperref}
\usepackage{color}
\usepackage{float}
\usepackage{graphics,amsmath,amssymb}
\usepackage{amsthm}
\usepackage{amsfonts}
\usepackage{latexsym}
\usepackage{epsf}
\usepackage{xifthen}
\usepackage{mathrsfs}
\usepackage{dsfont}
\usepackage{makecell}
\usepackage[FIGTOPCAP]{subfigure}
\usepackage{amsmath}
\allowdisplaybreaks[4]
\usepackage{listings}
\usepackage{etoolbox}
\usepackage{fancyhdr}
\usepackage{pdflscape}
\usepackage[title,toc,titletoc]{appendix}
\usepackage{enumitem}
\usepackage[noadjust]{cite}
\usepackage{forest}
\usepackage{cases}

 \newtheoremstyle{mytheorem}
 {3pt}
 {3pt}
 {\slshape}
 {}
 {\bfseries}
 {.}
 { }
 {}

\numberwithin{equation}{section}

\theoremstyle{theorem}
\newtheorem{theorem}{Theorem}[section]
\newtheorem*{theorem*}{Theorem}

\newtheorem{corollary}[theorem]{Corollary}
\newtheorem{lemma}[theorem]{Lemma}
\newtheorem{proposition}[theorem]{Proposition}

\newtheorem{innercustomgeneric}{\customgenericname}
\providecommand{\customgenericname}{}
\newcommand{\newcustomtheorem}[2]{%
	\newenvironment{#1}[1]
	{%
		\renewcommand\customgenericname{#2}%
		\renewcommand\theinnercustomgeneric{##1}%
		\innercustomgeneric
	}
	{\endinnercustomgeneric}
}
\newcustomtheorem{ctheorem}{Theorem}
\newcustomtheorem{clemma}{Lemma}

\theoremstyle{definition}
\newtheorem{definition}{Definition}[section]

\newtheorem{example}{Example}[section]
\newtheorem*{example*}{Example}
\newtheorem*{examples*}{Examples}
\newtheorem{remark}{Remark}[section]
\newtheorem*{remark*}{Remark}
\newtheorem*{remarks*}{Remarks}

\newtheoremstyle{named}{}{}{\itshape}{}{\bfseries}{.}{.5em}{#1\thmnote{ #3}}
\theoremstyle{named}

\newcommand{\doi}[1]{\href{http://dx.doi.org/#1}{doi: #1}}

\newcommand{\Keywords}[1]{\ifthenelse{\isempty{#1}}{}{\smallskip \smallskip \noindent \textbf{Keywords}. #1}}
\newcommand{\MSC}[2][2020]{\ifthenelse{\isempty{#2}}{}{\smallskip \smallskip \noindent \textbf{#1MSC}. #2}}
\newcommand{\abstractnote}[1]{\ifthenelse{\isempty{#1}}{}{\smallskip \smallskip \noindent \textsuperscript{\dag}#1}}

\makeatletter
\def\specialsection{\@startsection{section}{1}%
  \z@{\linespacing\@plus\linespacing}{.5\linespacing}%
  {\normalfont}}
\def\section{\@startsection{section}{1}%
  \z@{.7\linespacing\@plus\linespacing}{.5\linespacing}%
  {\normalfont\scshape}}
\patchcmd{\@settitle}{\uppercasenonmath\@title}{\Large\boldmath}{}{}
\patchcmd{\@settitle}{\begin{center}}{\begin{flushleft}}{}{}
\patchcmd{\@settitle}{\end{center}}{\end{flushleft}}{}{}
\patchcmd{\@setauthors}{\MakeUppercase}{\normalsize}{}{}
\patchcmd{\@setauthors}{\centering}{\raggedright}{}{}
\patchcmd{\section}{\scshape}{\large\bfseries\boldmath}{}{}
\patchcmd{\subsection}{\bfseries}{\bfseries\boldmath}{}{}
\renewcommand{\@secnumfont}{\bfseries}
\patchcmd{\@startsection}{\@afterindenttrue}{\@afterindentfalse}{}{}
\patchcmd{\abstract}{\leftmargin3pc}{\leftmargin1pc}{}{}

\def\maketitle{\par
  \@topnum\z@ 
  \@setcopyright
  \thispagestyle{empty}
  \ifx\@empty\shortauthors \let\shortauthors\shorttitle
  \else \andify\shortauthors
  \fi
  \@maketitle@hook
  \begingroup
  \@maketitle
  \toks@\@xp{\shortauthors}\@temptokena\@xp{\shorttitle}%
  \toks4{\def\\{ \ignorespaces}}
  \edef\@tempa{%
    \@nx\markboth{\the\toks4
      \@nx\MakeUppercase{\the\toks@}}{\the\@temptokena}}%
  \@tempa
  \endgroup
  \c@footnote\z@
  \@cleartopmattertags
}
\makeatother

\newcommand{\bB}{\mathbf{B}}

\newcommand{\cB}{\mathcal{B}}
\newcommand{\cS}{\mathcal{S}}

\newcommand{\cL}{\mathcal{L}}

\newcommand{\sI}{\mathscr{I}}

\newcommand{\fs}{\mathsf{s}}

\newcommand{\qbinom}[2]{\begin{bmatrix}#1\\#2\end{bmatrix}}

\title{Linked partition ideals and Euclidean billiard partitions}

\author[S. Chern]{Shane Chern}
\address{Department of Mathematics and Statistics, Dalhousie University, Halifax, Nova Scotia, B3H 4R2, Canada}
\email{chenxiaohang92@gmail.com}

\date{}

\begin{document}

\maketitle

\begin{abstract}

Euclidean billiard partitions were recently introduced by Andrews, Dragovi{\'c} and Radnovi{\'c} in their study of periodic trajectories of ellipsoidal billiards in the Euclidean space. They are integer partitions into distinct parts such that (E1) adjacent parts are never both odd; (E2) the smallest part is even. By refining the framework of linked partition ideals, we establish a couple of relevant trivariate generating function identities, from which the result of Andrews, Dragovi{\'c} and Radnovi{\'c} follows as an immediate consequence.


\MSC{05A15, 05A17, 11P84.}
\end{abstract}

\section{Introduction}

\emph{Euclidean billiard partitions} arose from the work of Dragovi{\'c} and Radnovi{\'c} \cite{DR2015} in which periodic trajectories of ellipsoidal billiards in the Euclidean space were studied. Briefly speaking, such partitions are constructed by the period of a periodic trajectory as the largest part, followed by a sequence of winding numbers as the remaining parts. Recently, Andrews, Dragovi{\'c} and Radnovi{\'c} \cite{ADR2020} further translated this geometric definition to the language of partition theory.

\begin{definition}\label{def:EE12}
	\emph{Euclidean billiard partitions} are integer partitions into distinct parts such that
	\begin{itemize}[align=left,
		leftmargin=2.5em,
		itemindent=0pt,
		labelsep=0pt,
		labelwidth=2.5em]
		\item[(E1)] adjacent parts are never both odd;
		\item[(E2)] the smallest part is even.
	\end{itemize}
	We denote by $\cB$ the set of Euclidean billiard partitions.
\end{definition}

Given any integer partition $\lambda$, let us adopt the convention that $|\lambda|$ and $\sharp(\lambda)$ denote the sum of all parts (namely, the \emph{size}) and the number of parts (namely, the \emph{length}) in $\lambda$, respectively. Also, we denote by $\sharp_{a,M}(\lambda)$ the number of parts in $\lambda$ that are congruent to $a$ modulo $M$.

In order to enumerate the number of possible choices of types of caustics for each billiard trajectory, Andrews, Dragovi{\'c} and Radnovi{\'c} assigned a weight $\varphi$ to each Euclidean billiard partition $\lambda$ as follows:
\begin{align*}
	\varphi(\lambda):=\begin{cases}
		\sharp(\lambda)-2\sharp_{1,2}(\lambda)-1 & \text{if the largest part in $\lambda$ is even},\\
		\sharp(\lambda)-2\sharp_{1,2}(\lambda) & \text{if the largest part in $\lambda$ is odd}.
	\end{cases}
\end{align*}
The main result in \cite{ADR2020} is the following bivariate generating function identity for Euclidean billiard partitions.

\begin{theorem}[Andrews--Dragovi{\'c}--Radnovi{\'c} {\cite[Theorem 2.8]{ADR2020}}]\label{th:ADR}
	We have
	\begin{align}\label{eq:ADR}
		1+\sum_{\lambda\in \cB} x^{\varphi(\lambda)}q^{|\lambda|} = 1+\sum_{d\ge 1}\sum_{m\ge 0}\frac{s(d,m)}{(q^2;q^2)_d}
	\end{align}
	where
	\begin{align*}
		s(d,m)=\begin{cases}
			x^{2n-d}q^{d^2+2n^2-2dn+3n}\qbinom{n-1}{2n-d}_{q^2} & \text{if $m=2n+1$},\\[20pt]
			x^{2n-d-1}q^{d^2+2n^2-2dn+2d-n}\qbinom{n-1}{2n-d-1}_{q^2} & \text{if $m=2n$},
		\end{cases}
	\end{align*}
	with the $q$-Pochhammer symbol defined for $n\in\mathbb{N}\cup\{\infty\}$ by
	\begin{align*}
		(A;q)_n:=\prod_{k=0}^{n-1}(1-A q^k)
	\end{align*}
	and the $q$-binomial coefficients defined by
	\begin{align*}
		\qbinom{A}{B}_q:=\begin{cases}
			\dfrac{(q;q)_A}{(q;q)_B(q;q)_{A-B}} & \text{if $0\le B\le A$},\\[12pt]
			0 & \text{otherwise}.
		\end{cases}
	\end{align*}
\end{theorem}

Looking at the geometric side, the weight $\varphi$ is of high significance. However, it is more natural to consider simply the enumerations of parts in certain arithmetic progressions (such as the number of odd parts appearing in the definition of $\varphi$) from a partition-theoretic perspective. For instance, one might be curious if there is an explicit expression for the generating function
$$B(x,y,q):=1+\sum_{\lambda\in \cB} x^{\sharp(\lambda)} y^{\sharp_{1,2}(\lambda)}q^{|\lambda|}.$$
Meanwhile, we may further separate the set of Euclidean billiard partitions into disjoint subsets according to the parity of the largest part.

\begin{definition}
	Let $\cB_E$ (resp.~$\cB_O$) denote the set of Euclidean billiard partitions with the largest part even (resp.~odd).
\end{definition}

As long as one can get nice expressions for the trivariate generating functions
\begin{align*}
	B_E(x,y,q)&:=\sum_{\lambda\in \cB_E} x^{\sharp(\lambda)} y^{\sharp_{1,2}(\lambda)}q^{|\lambda|},\\
	B_O(x,y,q)&:=\sum_{\lambda\in \cB_O} x^{\sharp(\lambda)} y^{\sharp_{1,2}(\lambda)}q^{|\lambda|},
\end{align*}
it is immediate that
\begin{align}\label{eq:ADR-alt}
	1+\sum_{\lambda\in \cB} x^{\varphi(\lambda)}q^{|\lambda|} = 1+ x^{-1}B_E(x,x^{-2},q)+B_O(x,x^{-2},q),
\end{align}
so the generating function identities for $B_E(x,y,q)$ and $B_O(x,y,q)$ as well as for $B(x,y,q)$ shall deliver more information.

Motivated by the above discussions, the first object of this paper concerns the following trivariate generating function identities.

\begin{theorem}\label{th:gf-EB}
	We have
	\begin{align}\label{eq:gf-EB}
		1+\sum_{\lambda\in \cB} x^{\sharp(\lambda)} y^{\sharp_{1,2}(\lambda)}q^{|\lambda|} = \sum_{i\ge 0}\sum_{j\ge 0}\frac{x^{i+j}y^jq^{i^2+j^2+i+2j}}{(q^2;q^2)_{i+j}}\qbinom{i}{j}_{q^2}.
	\end{align}
	Furthermore,
	\begin{align}\label{eq:gf-EB-E}
		\sum_{\lambda\in \cB_E} x^{\sharp(\lambda)} y^{\sharp_{1,2}(\lambda)}q^{|\lambda|} = \sum_{i\ge 1}\sum_{j\ge 0}\frac{x^{i+j}y^jq^{i^2+j^2+i+2j}}{(q^2;q^2)_{i+j}}\qbinom{i-1}{j}_{q^2}
	\end{align}
	and
	\begin{align}\label{eq:gf-EB-O}
		\sum_{\lambda\in \cB_O} x^{\sharp(\lambda)} y^{\sharp_{1,2}(\lambda)}q^{|\lambda|} = \sum_{i\ge 1}\sum_{j\ge 0}\frac{x^{i+j}y^jq^{i^2+j^2+3i}}{(q^2;q^2)_{i+j}}\qbinom{i-1}{j-1}_{q^2}.
	\end{align}
\end{theorem}

\begin{remark}
	Making the following change of variables
	\begin{align*}
		\left\{
		\begin{aligned}
			i&=n\\
			j&=d-n
		\end{aligned}
		\right. \qquad\iff\qquad
		\left\{
		\begin{aligned}
			d&=i+j\\
			n&=i
		\end{aligned}
		\right.
	\end{align*}
	in \eqref{eq:gf-EB-E} and \eqref{eq:gf-EB-O}, we obtain
	\begin{align*}
		x^{-1}B_E(x,x^{-2},q) = \sum_{d\ge 1}\sum_{n\ge 0}\frac{x^{2n-d-1}q^{d^2+2n^2-2dn+2d-n}}{(q^2;q^2)_d}\qbinom{n-1}{d-n}_{q^2}
	\end{align*}
	and
	\begin{align*}
		B_O(x,x^{-2},q) = \sum_{d\ge 1}\sum_{n\ge 0}\frac{x^{2n-d}q^{d^2+2n^2-2dn+3n}}{(q^2;q^2)_d}\qbinom{n-1}{d-n-1}_{q^2}.
	\end{align*}
	Recalling \eqref{eq:ADR-alt} immediately yields the result of Andrews, Dragovi{\'c} and Radnovi{\'c} in \eqref{eq:ADR}, while we shall also invoke a trivial property of the $q$-binomial coefficients \cite[p.~35, (3.3.2)]{And1976}:
	$$\qbinom{A}{B}_q = \qbinom{A}{A-B}_q.$$
\end{remark}

On the other hand, it is plain that Condition (E2) in Definition \ref{def:EE12} only constrains the parity of the smallest part in the partition in question. Therefore, Euclidean billiard partitions form a subset of distinct partitions that are only restricted by Condition (E1).

\begin{definition}\label{def:EE}
	We denote by $\cS$ the set of integer partitions into distinct parts such that
	\begin{itemize}[align=left,
		leftmargin=2.5em,
		itemindent=0pt,
		labelsep=0pt,
		labelwidth=2.5em]
		\item[(E1)] adjacent parts are never both odd.
	\end{itemize}
	Furthermore, let $\cS_E$ (resp.~$\cS_O$) denote the set of partitions in $\cS$ with the largest part even (resp.~odd).
\end{definition}

We have parallel results for these partitions.

\begin{theorem}\label{th:gf-S}
	We have
	\begin{align}\label{eq:gf-S}
		1+\sum_{\lambda\in \cS} x^{\sharp(\lambda)} y^{\sharp_{1,2}(\lambda)}q^{|\lambda|} = \sum_{i\ge 0}\sum_{j\ge 0}\frac{x^{i+j}y^jq^{i^2+j^2+i}}{(q^2;q^2)_{i+j}}\qbinom{i+1}{j}_{q^2}.
	\end{align}
	Furthermore,
	\begin{align}\label{eq:gf-S-E}
		\sum_{\lambda\in \cS_E} x^{\sharp(\lambda)} y^{\sharp_{1,2}(\lambda)}q^{|\lambda|} = \sum_{i\ge 1}\sum_{j\ge 0}\frac{x^{i+j}y^jq^{i^2+j^2+i}}{(q^2;q^2)_{i+j}}\qbinom{i}{j}_{q^2}
	\end{align}
	and
	\begin{align}\label{eq:gf-S-O}
		\sum_{\lambda\in \cS_O} x^{\sharp(\lambda)} y^{\sharp_{1,2}(\lambda)}q^{|\lambda|} = \sum_{i\ge 1}\sum_{j\ge 0}\frac{x^{i+j}y^{j+1}q^{i^2+j^2+i-1}}{(q^2;q^2)_{i+j}}\qbinom{i-1}{j}_{q^2}.
	\end{align}
\end{theorem}

To establish Theorem \ref{th:ADR}, Andrews, Dragovi{\'c} and Radnovi{\'c} utilized the technique of \emph{separable integer partition classes}, which was later systematically elaborated by Andrews in \cite{And2022}. However, our derivation of the above trivariate generating function identities relies on a different approach by absorbing ideas from \emph{linked partition ideals} introduced by Andrews \cite{And1972,And1974b,And1975} in the 1970s and reflourished by Chern and Li \cite{CL2018} quite recently. In particular, a major advantage of the framework of linked partition ideals is that it allows us to freely insert new parameters to count additional partition statistics that are closely related to the shape of the partitions in question.

In Sect.~\ref{sec:ref-lpi}, we refine the framework of linked partition ideals, which will be used in Sect.~\ref{sec:EBP} for the combinatorial constructions on Euclidean billiard partitions. We shall then derive a system of $q$-difference equations involving the desired generating functions. By solving this $q$-difference system with corresponding boundary conditions (i.e.~initial coefficients of the power series solutions), we prove Theorems \ref{th:gf-S} and \ref{th:gf-EB} in Sects.~\ref{sec:gf-S} and \ref{sec:gf-EB}, respectively. We close this paper with a conclusion in Sect.~\ref{sec:rmk}.

\section{Refined span one linked partition ideals}\label{sec:ref-lpi}

In the study of partitions constrained by conditions on the difference of neighboring parts such as partitions related to Schur-type identities \cite{ACL2021} or partitions arising from the Kanade--Russell conjectures \cite{CL2018}, a special type of linked partition ideals, known as \emph{span one linked partition ideals}, is of particular importance. However, to fit Euclidean billiard partitions into this framework, we have to first make a couple of refinements. It is necessary to point out in advance that this section only covers some generic definitions and the concrete example for Euclidean billiard partitions will be presented in the next section.

\begin{definition}[Refined span one linked partition ideals]
	Assume that we are given
	\begin{itemize}[leftmargin=*,align=left]
		\renewcommand{\labelitemi}{\scriptsize$\blacktriangleright$}
		
		\item a finite set $\Pi=\{\pi_1,\ldots,\pi_J,\pi_{J+1},\ldots,\pi_{J+K}\}$ of integer partitions where $\pi_1=\cdots=\pi_J=\varnothing$, the empty partition, while we deliberately assume that they are different, and $\pi_{J+1},\ldots,\pi_{J+K}$ are different nonempty partitions;
		
		\item a \textit{map of linking sets}, $\cL:\Pi\to P(\Pi)$, the power set of $\Pi$, such that for $1\le j\le J$, $\pi_j\in \cL(\pi_j)$ and $\pi_{j'}\not\in \cL(\pi_j)$ whenever $1\le j'\le J$ and $j'\ne j$, and that for $1\le k\le K$, there is \textbf{exactly} one $\pi_j$ with $1\le j\le J$ such that $\pi_j\in \cL(\pi_{J+k})$;
		
		\item and a positive integer $T$ called the \textit{modulus} such that it is no smaller than the largest part among all partitions in $\Pi$.
	\end{itemize}
	We say a \textit{refined span one linked partition ideal} $\sI=\sI(\langle\Pi,\cL\rangle,T)$ is the \textbf{multiset} of partitions of the form
	\begin{align}\label{eq:decomp}
		\lambda&=\phi^0(\lambda_0)\oplus \phi^T(\lambda_1)\oplus \cdots \oplus \phi^{NT}(\lambda_N)\oplus \phi^{(N+1)T}(\pi_j)\oplus \phi^{(N+2)T}(\pi_j)\oplus \cdots\notag\\
		&=\phi^0(\lambda_0)\oplus \phi^T(\lambda_1)\oplus \cdots \oplus \phi^{NT}(\lambda_N),
	\end{align}
	where all $\lambda_i$ are from $\Pi$ with $\lambda_i\in\cL(\lambda_{i-1})$ for each $i$, and $\lambda_N$ is not the empty partition while the empty partition $\pi_j$ with $1\le j\le J$ is such that $\pi_j\in \cL(\lambda_N)$. We also include $J$ copies of the empty partition in $\sI$, and they correspond to 
	\begin{align*}
		\varnothing&=\phi^{0}(\pi_1)\oplus \phi^{T}(\pi_1)\oplus \cdots\\
		&=\cdots\\
		&=\phi^{0}(\pi_J)\oplus \phi^{T}(\pi_J)\oplus \cdots.
	\end{align*}
	Here for any two partitions $\mu$ and $\nu$, $\mu\oplus\nu$ gives a partition by collecting all parts in $\mu$ and $\nu$, and $\phi^m(\mu)$ gives a partition by adding $m$ to each part of $\mu$.
\end{definition}

\begin{remark}
	The original definition of span one linked partition ideals (see, e.g.~\cite[Definition 2.1]{ACL2021}) corresponds to the $J=1$ case. However, we also slightly loosen the requirements for the map of linking sets for generality --- in the original definition, all nonempty partitions in $\Pi$ are assumed to be in $\cL(\pi_1)$, where $\pi_1$ is the only empty partition in $\Pi$ in the $J=1$ case.
\end{remark}

The underlying logic of (refined) span one linked partition ideals is that for every partition $\lambda\in \sI$, we may decompose the parts into blocks $\bB_0, \bB_1, \ldots$ such that all parts between $Ti+1$ and $Ti+T$ fall into the block $\bB_i$. Now applying the operator $\phi^{-Ti}$ to each block, we get a list of partitions with the largest part at most $T$, and they are all in $\Pi$.

We shall call the partition $\phi^{-Ti}(\bB_i)$ the \emph{prototype} of the block $\bB_i$.

If we denote by $\lambda_i$ the prototype of the block $\bB_i$, then we get a finite chain of partitions in $\Pi$, i.e.~$\lambda_0\to\lambda_1\to \cdots \to\lambda_N$ (where $\lambda_N$ corresponds to the last nonempty block), which can be further extended as an infinite chain ending with a series of empty partitions $\pi_j\to\pi_j\to\cdots$ such that $\pi_j\in \cL(\lambda_N)$ where $1\le j\le J$. In particular, the correspondence of
\begin{gather*}
	\lambda\\
	\Updownarrow\\
	\lambda_0\to\lambda_1\to \cdots \to\lambda_N\to \pi_j\to\pi_j\to\cdots
\end{gather*}
is connected by \eqref{eq:decomp}. By abuse of notation, we will not distinguish a partition $\lambda$ in $\sI$ and its linked partition ideal decomposition $\lambda_0\lambda_1\lambda_2\cdots$, which is short for the chain $\lambda_0\to\lambda_1\to \lambda_2\to\cdots$, and for convenience, we simply write $\lambda=\lambda_0\lambda_1\lambda_2\cdots$.

Recall that the refined span one linked partition ideal $\sI=\sI(\langle\Pi,\cL\rangle,T)$ is a multiset of partitions. For instance, the multiplicity of the empty partition $\varnothing$ is $J$ as we have
$$\varnothing = \pi_j\pi_j\cdots$$
for every $1\le j\le J$. Now a crucial question concerns the multiplicity of nonempty partitions in $\sI$.

\begin{lemma}\label{le:rlpi-multi}
	Let $\lambda$ be a nonempty partition in $\sI=\sI(\langle\Pi,\cL\rangle,T)$. Assume that $\bB_M$ is the first nonempty block with prototype $\lambda_M$.
	\begin{enumerate}[label={\textup{(\arabic*).~}},leftmargin=*,labelsep=0cm,align=left]
		\item If $M=0$, then the multiplicity of $\lambda$ in $\sI$ is $1$, and we have $\lambda = \lambda_0\lambda_1\cdots$.
		
		\item If $M>0$, and assume that there are $\alpha$ empty partitions among $\pi_1,\ldots,\pi_J$, given by $\pi_{j_1},\ldots,\pi_{j_{\alpha}}$, such that $\lambda_M$ is in $\cL(\pi_{j_1}),\ldots,\cL(\pi_{j_{\alpha}})$, then the multiplicity of $\lambda$ in $\sI$ is $\alpha$, and we have
		$$\lambda = \pi_{j_{a}}\pi_{j_{a}}\cdots \pi_{j_{a}}\lambda_M\lambda_{M+1}\cdots,$$
		for each $1\le a\le \alpha$.
	\end{enumerate}
\end{lemma}

\begin{proof}
	Since the block $\bB_M$ is nonempty, so is its prototype $\lambda_M$ and we see that $\lambda_M$ is uniquely given by one of $\pi_{J+1},\ldots,\pi_{J+K}$. Now if the successor of $\lambda_M$, namely, $\lambda_{M+1}$ is nonempty, it is also uniquely given by one of $\pi_{J+1},\ldots,\pi_{J+K}$; if $\lambda_{M+1}$ is the empty partition, then it is still uniquely determined as there is exactly one of the empty partitions $\pi_1,\ldots,\pi_J$ contained in $\cL(\lambda_M)$. Continuing this process, we find that all $\lambda_m$ with $m\ge M$ are uniquely determined.
	
	It remains to characterize the predecessor of $\lambda_M$. If $M=0$, then there is no predecessor of $\lambda_M$ and hence $\lambda$ is uniquely given by $\lambda_0\lambda_1\cdots$. If $M>0$, then by our assumption, $\lambda_{M-1}$ is one of the empty partitions $\pi_1,\ldots,\pi_J$, say $\pi_j$. Furthermore, we must have $\lambda_M\in \cL(\pi_j)$; otherwise, $\lambda_M$ cannot be the successor of $\pi_j$. Finally, we note that the empty partition predecessor of $\pi_j$ could only be $\pi_j$ itself as for $1\le j'\le J$ with $j'\ne j$, we have $\pi_j\not\in \cL(\pi_{j'})$. In other words, $\lambda_M$ is preceded by $\pi_j\pi_j\cdots \pi_j$, as required.
\end{proof}

\begin{definition}\label{def:I-subset}
	Let $1\le i\le J+K$ and $1\le j\le J$ be given indices, we denote by $\sI_{i,j}$ the following subset of $\sI=\sI(\langle\Pi,\cL\rangle,T)$:
	$$\sI_{i,j}:=\{\lambda\in\sI:\text{$\lambda_0=\pi_i$ and $\lambda$ ends with $\pi_j\pi_j\cdots$}\}.$$
	We further define
	$$\sI_{i}=\sI_{i,0}:=\{\lambda\in\sI:\lambda_0=\pi_i\}.$$
\end{definition}

An immediate consequence of Lemma \ref{le:rlpi-multi} is as follows.

\begin{corollary}
	Let $1\le i\le J+K$, $1\le j\le J$ and $1\le k,k'\le K$ be indices. Then
	\begin{enumerate}[label={\textup{(\arabic*).~}},leftmargin=*,labelsep=0cm,align=left]
		\item No partition repeats in $\sI_i$.
		
		\item No partition is contained in both $\sI_j$ and $\sI_{J+k}$, and no partition is contained in both $\sI_{J+k}$ and $\sI_{J+k'}$ whenever $k\ne k'$.
	\end{enumerate}
\end{corollary}

Now we consider the related generating functions.

\begin{definition}\label{def:I-gf}
	Assume that $\fs$ is a partition statistic such that for $\lambda\in\sI$ with the linked partition ideal decomposition $\lambda = \lambda_0\lambda_1\cdots$,
	$$\fs(\lambda)=\fs(\lambda_0)+\fs(\lambda_1)+\cdots.$$
	Define for $1\le i\le J+K$ and $1\le j\le J$,
	$$H_{i,j}(x) = H_{i,j}(x,y,q):= \sum_{\lambda\in \sI_{i,j}}x^{\sharp(\lambda)}y^{\fs(\lambda)}q^{|\lambda|}.$$
	We further define
	$$H_{i}(x)=H_{i,0}(x)=H_{i,0}(x,y,q):=\sum_{\lambda\in \sI_{i}}x^{\sharp(\lambda)}y^{\fs(\lambda)}q^{|\lambda|}.$$
\end{definition}

We shall establish the following relations satisfied by these generating functions. In particular, if the relations are listed explicitly, we are led to a system of $q$-difference equations.

\begin{theorem}\label{th:gf-sys}
	For $1\le i\le J+K$ and $0\le j\le J$,
	\begin{align}
		H_{i,j}(x) = x^{\sharp(\pi_i)}y^{\fs(\pi_i)}q^{|\pi_i|}\sum_{i':\pi_{i'}\in \cL(\pi_i)} H_{i',j}(xq^T).
	\end{align}
\end{theorem}

\begin{proof}
	Given any partition $\lambda\in \sI_{i,j}$, if we write in terms of the linked partition ideal decomposition $\lambda = \lambda_0\lambda_1\cdots$, then $\lambda_0=\pi_i$. Note that
	\begin{align*}
		\lambda &= \phi^0(\lambda_0)\oplus \phi^T(\lambda_1)\oplus \phi^{2T}(\lambda_2)\oplus \cdots \\
		&= \lambda_0 \oplus \phi^T\big(\lambda_1\oplus \phi^T(\lambda_2)\oplus\cdots\big)\\
		&= \pi_i \oplus \phi^T(\lambda'),
	\end{align*}
	where $\lambda' = \lambda_1\oplus \phi^T(\lambda_2)\oplus\cdots$. In particular, $\lambda'\in \bigsqcup_{i':\pi_{i'}\in \cL(\pi_i)}\sI_{i',j}$ as $\lambda_1\in\cL(\lambda_0)= \cL(\pi_i)$. Further,
	\begin{align*}
		\fs(\lambda) = \fs(\lambda_0)+\fs(\lambda_1)+\fs(\lambda_2)+\cdots = \fs(\pi_i)+\fs(\lambda').
	\end{align*}
	Noting that $\lambda'\in \bigsqcup_{i':\pi_{i'}\in \cL(\pi_i)}\sI_{i',j}$ bijectively correspond to $\lambda\in \sI_{i,j}$, we conclude that
	\begin{align*}
		H_{i,j}(x) &= \sum_{\lambda\in \sI_{i,j}}x^{\sharp(\lambda)}y^{\fs(\lambda)}q^{|\lambda|}\\
		&=\sum_{i':\pi_{i'}\in \cL(\pi_i)} \sum_{\lambda'\in \sI_{i',j}} x^{\sharp(\pi_i)+\sharp(\lambda')} y^{\fs(\pi_i)+\fs(\lambda')} q^{|\pi_i|+|\lambda'|+T\cdot \sharp(\lambda')}\\
		&=x^{\sharp(\pi_i)} y^{\fs(\pi_i)} q^{|\pi_i|} \sum_{i':\pi_{i'}\in \cL(\pi_i)} \sum_{\lambda'\in \sI_{i',j}} (xq^T)^{\sharp(\lambda')} y^{\fs(\lambda')} q^{|\lambda'|}\\
		&= x^{\sharp(\pi_i)}y^{\fs(\pi_i)}q^{|\pi_i|}\sum_{i':\pi_{i'}\in \cL(\pi_i)} H_{i',j}(xq^T),
	\end{align*}
	as requested.
\end{proof}

\section{Euclidean billiard partitions}\label{sec:EBP}

Following the framework in the previous section, we choose
\begin{align*}
	\Pi&=\{\pi_1,\pi_2,\pi_3,\pi_4,\pi_5\}\\
	&=\{\varnothing_E,\varnothing_O,1,2,1+2\},
\end{align*}
where both $\varnothing_E$ and $\varnothing_O$ are the empty partition while they are deliberately assumed to be different. Also, the map of linking sets is given by
\begin{equation*}
	\begin{array}{lp{0.5cm}l}
		\pi\in\Pi && \cL(\pi)\\
		\hline\rule{0pt}{1\normalbaselineskip}
		\pi_1 = \varnothing_E && \{\pi_1,\pi_3,\pi_4,\pi_5\} = \{\varnothing_E,1,2,1+2\}\\
		\pi_2 = \varnothing_O && \{\pi_2,\pi_4\} = \{\varnothing_O,2\}\\
		\pi_3 = 1 && \{\pi_2,\pi_4\} = \{\varnothing_O,2\}\\
		\pi_4 = 2 && \{\pi_1,\pi_3,\pi_4,\pi_5\} = \{\varnothing_E,1,2,1+2\}\\
		\pi_5 = 1+2 && \{\pi_1,\pi_3,\pi_4,\pi_5\} = \{\varnothing_E,1,2,1+2\}
	\end{array}
\end{equation*}
Finally, the modulus is chosen by $T=2$.

Recalling Definitions \ref{def:I-subset} and \ref{def:I-gf}, we shall consider the generating functions for $1\le i\le 5$ and $1\le j\le 2$,
$$H_{i,j}(x) := \sum_{\lambda\in \sI_{i,j}}x^{\sharp(\lambda)}y^{\sharp_{1,2}(\lambda)}q^{|\lambda|},$$
and additionally,
$$H_{i}(x)=H_{i,0}(x):=\sum_{\lambda\in \sI_{i}}x^{\sharp(\lambda)}y^{\sharp_{1,2}(\lambda)}q^{|\lambda|}.$$
By Theorem \ref{th:gf-sys}, we have the following $q$-difference system.

\begin{proposition}
	For $0\le j\le 2$,
	\begin{subnumcases}{}
		H_{1,j}(x) = H_{1,j}(xq^2)+H_{3,j}(xq^2)+H_{4,j}(xq^2)+H_{5,j}(xq^2),\label{eq:q-sys-1-1}\\
		H_{2,j}(x) = H_{2,j}(xq^2)+H_{4,j}(xq^2),\label{eq:q-sys-1-2}\\
		H_{3,j}(x) = xyq\big(H_{2,j}(xq^2)+H_{4,j}(xq^2)\big),\label{eq:q-sys-1-3}\\
		H_{4,j}(x) = xq^2\big(H_{1,j}(xq^2)+H_{3,j}(xq^2)+H_{4,j}(xq^2)+H_{5,j}(xq^2)\big),\label{eq:q-sys-1-4}\\
		H_{5,j}(x) = x^2yq^3\big(H_{1,j}(xq^2)+H_{3,j}(xq^2)+H_{4,j}(xq^2)+H_{5,j}(xq^2)\big).\label{eq:q-sys-1-5}
	\end{subnumcases}
\end{proposition}

Note that by \eqref{eq:q-sys-1-1}, \eqref{eq:q-sys-1-4} and \eqref{eq:q-sys-1-5},
\begin{align*}
	\begin{cases}
		H_{4,j}(x)=xq^2H_{1,j}(x),\\
		H_{5,j}(x)=x^2yq^3H_{1,j}(x),
	\end{cases}
\end{align*}
and that by \eqref{eq:q-sys-1-2} and \eqref{eq:q-sys-1-3},
\begin{align*}
	H_{3,j}(x)=xyqH_{2,j}(x).
\end{align*}
Therefore, the above $q$-difference system can be simplified as follows.

\begin{corollary}
	For $0\le j\le 2$,
	\begin{subnumcases}{}
		H_{1,j}(x) = (1+xq^4+x^2yq^7)H_{1,j}(xq^2)+xyq^3H_{2,j}(xq^2),\label{eq:q-sys-2-1}\\
		H_{2,j}(x) = xq^4H_{1,j}(xq^2)+H_{2,j}(xq^2).\label{eq:q-sys-2-2}
	\end{subnumcases}
\end{corollary}

Recall the conditions
\begin{itemize}[align=left,
	leftmargin=2.5em,
	itemindent=0pt,
	labelsep=0pt,
	labelwidth=2.5em]
	\item[(E1)] adjacent parts are never both odd;
	\item[(E2)] the smallest part is even.
\end{itemize}
Partitions in $\cB$ (i.e.~Euclidean billiard partitions) are partitions into distinct parts satisfying both (E1) and (E2); partitions in $\cS$ are partitions into distinct parts satisfying only (E1).

It is plain that partitions in $\sI=\sI(\langle\Pi,\cL\rangle,T)$ are partitions into distinct parts.

Now we explain why we assign the subscripts ``$E$'' and ``$O$'' in $\pi_1=\varnothing_E$ and $\pi_2=\varnothing_O$. Let $\lambda\in \sI$. Assume that in the linked partition ideal decomposition $\lambda_0\lambda_1\cdots$ of $\lambda$, there is a node, say $\lambda_M$, equals $\varnothing_E$. If at least one nonempty partition precedes $\lambda_M$, and assume that $\lambda_{M'}$ is the last nonempty partition among $\lambda_0,\ldots,\lambda_{M-1}$, then $\varnothing_E\in \cL(\lambda_{M'})$ so that $\lambda_{M'}\in\{\pi_4,\pi_5\}=\{2,1+2\}$. In other words, the largest part in $\lambda$ that precedes the block $\bB_M$ is $2M'+2$, which is even. Similarly, when $\lambda_M=\varnothing_O$, the last nonempty partition $\lambda_{M'}$ preceding $\lambda_M$, if exists, is such that $\lambda_{M'}\in\{\pi_3\}=\{1\}$ so that the largest part in $\lambda$ that precedes the block $\bB_M$ is $2M'+1$, which is odd. In conclusion, the subscript ``$E$'' or ``$O$'' records the parity of the largest part preceding an empty block in the decomposition.

Meanwhile, when $\lambda_M=\varnothing_E$, the first nonempty partition $\lambda_{M''}$ succeeding $\lambda_M$, if exists, is such that $\lambda_{M''}\in \cL(\varnothing_E)$ so that $\lambda_{M''}\in \{\pi_3,\pi_4,\pi_5\} = \{1,2,1+2\}$. When $\lambda_M=\varnothing_O$, the first nonempty partition $\lambda_{M''}$ succeeding $\lambda_M$, if exists, is such that $\lambda_{M''}\in \cL(\varnothing_O)$ so that $\lambda_{M''}\in \{\pi_4\} = \{2\}$. Therefore, in the decomposition of $\lambda$, the segments
$$\pi_3\to\varnothing\to\varnothing\to\cdots\to\varnothing\to \pi_3\qquad\text{and}\qquad \pi_3\to\varnothing\to\varnothing\to\cdots\to\varnothing\to \pi_5,$$
or equivalently in $\lambda$, the consecutive parts
$$(2M'+1)+(2M''+1) \qquad\text{and}\qquad (2M'+1)+(2M''+1)+(2M''+2)$$
are forbidden. As a consequence, for partitions in $\sI$, Condition (E1) is satisfied.

Conversely, we may naturally decompose partitions in $\cS\supset \cB$ as refined span one linked partition ideals. However, it should be emphasized that if the decomposition of $\lambda\in \cS\supset \cB$ is such that the first nonempty block has prototype $\pi_4=2$ while it is not the leading block $\bB_0$, then as suggested by Lemma \ref{le:rlpi-multi}, $\lambda$ has exactly two decompositions in $\sI$:
$$\pi_1\pi_1\cdots \pi_1 \pi_4\cdots \qquad\text{and}\qquad \pi_2\pi_2\cdots \pi_2 \pi_4\cdots.$$

\begin{example}
	\textbf{(i).}~We decompose the partition $1+2+3+8+9+10$ in $\cS$ as
	$$\pi_5\pi_3\pi_2\pi_4\pi_5\pi_1\pi_1\cdots$$
	since
	\begin{align*}
		1+2+3+8+9+10 &= \phi^0(1+2)\oplus \phi^2(1) \oplus \phi^4(\varnothing) \oplus \phi^6(2)\\
		&\quad \oplus \phi^8(1+2) \oplus \phi^{10}(\varnothing) \oplus \phi^{12}(\varnothing) \oplus \cdots.
	\end{align*}
	\textbf{(ii).}~We decompose the partition $6+7+8+11+14+15$ in $\cS$ (and also in $\cB$) as
	$$\pi_1\pi_1\pi_4\pi_5\pi_1\pi_3\pi_4\pi_3\pi_2\pi_2\cdots \qquad\text{or}\qquad \pi_2\pi_2\pi_4\pi_5\pi_1\pi_3\pi_4\pi_3\pi_2\pi_2\cdots$$
	since
	\begin{align*}
		6+7+8+11+14+15&= \phi^0(\varnothing)\oplus \phi^2(\varnothing) \oplus \phi^4(2) \oplus \phi^6(1+2) \oplus \phi^8(\varnothing)\\
		&\quad \oplus \phi^{10}(1) \oplus \phi^{12}(2) \oplus \phi^{14}(1) \oplus \phi^{16}(\varnothing) \oplus \phi^{18}(\varnothing) \oplus \cdots.
	\end{align*}
\end{example}

The above discussions can be summarized as follows.

\begin{proposition}\label{prop:ptn-set-new}
	The following statements are true:
	\begin{enumerate}[label={\textup{(\arabic*).~}},leftmargin=*,labelsep=0cm,align=left]
		\item Partitions in $\cS$ have a bijective correspondence with nonempty partitions in $\sI_{1}\sqcup \sI_{3}\sqcup \sI_{4}\sqcup \sI_{5}$. Further, partitions in $\cS$ with largest part even (resp.~odd) have a bijective correspondence with nonempty partitions in $\sI_{1,1}\sqcup \sI_{3,1}\sqcup \sI_{4,1}\sqcup \sI_{5,1}$ (resp.~$\sI_{1,2}\sqcup \sI_{3,2}\sqcup \sI_{4,2}\sqcup \sI_{5,2}$).
		
		\item Partitions in $\cB$ have a bijective correspondence with nonempty partitions in $\sI_{2}\sqcup \sI_{4}$. Further, partitions in $\cS$ with largest part even (resp.~odd) have a bijective correspondence with nonempty partitions in $\sI_{2,1}\sqcup \sI_{4,1}$ (resp.~$\sI_{2,2}\sqcup \sI_{4,2}$).
	\end{enumerate}
\end{proposition}

\section{Proof of Theorem \ref{th:gf-S}}\label{sec:gf-S}

Now we establish the generating function identities in Theorem \ref{th:gf-S}. Note that once we have confirmed \eqref{eq:gf-S} and \eqref{eq:gf-S-E}, it is immediate that
\begin{align*}
	\sum_{\lambda\in \cS_O} x^{\sharp(\lambda)} y^{\sharp_{1,2}(\lambda)}q^{|\lambda|} &= \sum_{\lambda\in \cS} x^{\sharp(\lambda)} y^{\sharp_{1,2}(\lambda)}q^{|\lambda|} - \sum_{\lambda\in \cS_E} x^{\sharp(\lambda)} y^{\sharp_{1,2}(\lambda)}q^{|\lambda|}\\
	&=\sum_{i\ge 0}\sum_{j\ge 0}\frac{x^{i+j}y^jq^{i^2+j^2+i}}{(q^2;q^2)_{i+j}}\left(\qbinom{i+1}{j}_{q^2}-\qbinom{i}{j}_{q^2}\right)\\
	&=\sum_{i\ge 0}\sum_{j\ge 1}\frac{x^{i+j}y^jq^{i^2+j^2+i}}{(q^2;q^2)_{i+j}}\cdot q^{2(i+1-j)}\qbinom{i}{j-1}_{q^2}\\
	&=\sum_{i\ge 0}\sum_{j\ge 1}\frac{x^{i+j}y^jq^{i^2+j^2+3i-2j+2}}{(q^2;q^2)_{i+j}}\qbinom{i}{j-1}_{q^2}\\
	\text{\tiny($(i,j)\mapsto (i-1,j+1)$)}&=\sum_{i\ge 1}\sum_{j\ge 0}\frac{x^{i+j}y^{j+1}q^{i^2+j^2+i-1}}{(q^2;q^2)_{i+j}}\qbinom{i-1}{j}_{q^2},
\end{align*}
thereby implying \eqref{eq:gf-S-O}. Here we make use of the following relation for $q$-binomial coefficients \cite[p.~35, (3.3.3)]{And1976}: for $(A,B)\ne (0,0)$,
\begin{align}\label{eq:q-binom-1}
	\qbinom{A}{B}_q = \qbinom{A-1}{B}_q + q^{A-B} \qbinom{A-1}{B-1}_q.
\end{align}
Now it suffices to prove \eqref{eq:gf-S} and \eqref{eq:gf-S-E}.

\subsection{Proof of (\ref{eq:gf-S})}

By Proposition \ref{prop:ptn-set-new},
\begin{align*}
	1+\sum_{\lambda\in \cS} x^{\sharp(\lambda)} y^{\sharp_{1,2}(\lambda)}q^{|\lambda|} = H_{1}(x)+ H_{3}(x) + H_{4}(x) + H_{5}(x).
\end{align*}
Recalling \eqref{eq:q-sys-1-1} further gives
\begin{align}\label{eq:S-gf-H1}
	1+\sum_{\lambda\in \cS} x^{\sharp(\lambda)} y^{\sharp_{1,2}(\lambda)}q^{|\lambda|} = H_{1}(xq^{-2}).
\end{align}
For convenience, let us write
\begin{align*}
	S(x):=H_{1}(xq^{-2}).
\end{align*}

First, it follows from \eqref{eq:q-sys-2-1} that
\begin{align*}
	H_{2}(xq^2) = x^{-1}y^{-1}q^{-3}\big(H_{1}(x)-(1+xq^4+x^2yq^7)H_{1}(xq^2)\big).
\end{align*}
Substituting the above into \eqref{eq:q-sys-2-2} yields
\begin{align*}
	q^2 H_{1}(xq^{-2})-(1+q^2+xq^4+x^2yq^5)H_{1}(x) + (1+xq^4)H_{1}(xq^2) = 0.
\end{align*}
Namely,
\begin{align}\label{eq:F-q-dif}
	q^2 S(x)-(1+q^2+xq^4+x^2yq^5)S(xq^2) + (1+xq^4)S(xq^4) = 0.
\end{align}

We then write
\begin{align*}
	S(x) := \sum_{n\ge 0}s_n x^n.
\end{align*}
Apparently,
\begin{align}
	s_0&=1,\label{eq:s0}\\
	s_1&=yq+q^2+yq^3+q^4+\cdots = \frac{q(y+q)}{1-q^2}.\label{eq:s1}
\end{align}
We may translate the $q$-difference equation \eqref{eq:F-q-dif} into a recurrence of $s_n$: for $n\ge 2$,
\begin{align}\label{eq:s-rec}
	(1-q^{2n})(1-q^{2n-2})s_n = q^{2n} (1-q^{2n-2}) s_{n-1} + yq^{2n-1} s_{n-2}.
\end{align}

Now let us define, for $n\ge 0$,
\begin{align}\label{eq:s-t}
	t_n := s_n (q^2;q^2)_n.
\end{align}
Then by \eqref{eq:s0} and \eqref{eq:s1},
\begin{align*}
	t_0&=1,\\
	t_1&=q(y+q).
\end{align*}
Further, for $n\ge 2$, \eqref{eq:s-rec} becomes
\begin{align*}
	t_n = q^{2n}t_{n-1}+yq^{2n-1}t_{n-2}.
\end{align*}
If we write
\begin{align*}
	T(x):=\sum_{n\ge 0}t_n x^n,
\end{align*}
then
\begin{align*}
	T(x)-1-xq(y+q) = xq^2\big(T(xq^2)-1\big) + x^2 y q^3 T(xq^2),
\end{align*}
that is,
\begin{align}\label{eq:T}
	T(x)-xq^2(1+xyq)T(xq^2) = 1+xyq.
\end{align}

Here we give an explicit expression of $T(x)$.
\begin{lemma}
	We have
	\begin{align}\label{eq:T-exp}
		T(x)=\sum_{i\ge 0}x^i q^{i(i+1)}(-xyq;q^2)_{i+1}.
	\end{align}
\end{lemma}

\begin{proof}
	Let us prove a truncated result: for $N\ge 1$,
	\begin{align}\label{eq:T-trun}
		T(x)-x^Nq^{N(N+1)}(-xyq;q^2)_N T(xq^{2N}) = \sum_{i=0}^{N-1} x^i q^{i(i+1)} (-xyq;q^2)_{i+1}.
	\end{align}
	We shall see that our lemma follows by letting $N\to\infty$.
	
	To show \eqref{eq:T-trun}, we induct on $N$. First, the base case $N=1$ is exactly \eqref{eq:T}. Now assume that \eqref{eq:T-trun} is valid for some $N\ge 1$. Replacing $x$ by $xq^{2N}$ in \eqref{eq:T} gives
	\begin{align*}
		T(xq^{2N})-xq^{2N+2}(1+xyq^{2N+1})T(xq^{2N+2}) = 1+xyq^{2N+1}.
	\end{align*}
	Multiplying by $x^Nq^{N(N+1)}(-xq;q^2)_N$ on both sides of the above, and then combining the resulting relation with \eqref{eq:T-trun}, we have
	\begin{align*}
		T(x)-x^{N+1}q^{(N+1)(N+2)}(-xyq;q^2)_{N+1}T(xq^{2N+2}) = \sum_{i=0}^{N} x^i q^{i(i+1)} (-xyq;q^2)_{i+1}.
	\end{align*}
	This is exactly the $N+1$ case of \eqref{eq:T-trun} and therefore the desired result holds.
\end{proof}

Finally, to deduce an explicit expression for $S(x)$, we need to reformulate $T(x)$. We require the $q$-binomial theorem \cite[p.~36, (3.3.6)]{And1976}: for $n\ge 0$,
\begin{align}\label{eq:q-binom}
	(z;q)_n = \sum_{j\ge 0} \qbinom{n}{j}_q (-1)^j q^{\binom{j}{2}}z^j.
\end{align}
By \eqref{eq:T-exp},
\begin{align*}
	T(x)&=\sum_{i\ge 0}x^i q^{i(i+1)}(-xyq;q^2)_{i+1}\\
	\text{\tiny(by \eqref{eq:q-binom})}&=\sum_{i\ge 0}x^i q^{i(i+1)}\sum_{j\ge 0}\qbinom{i+1}{j}_{q^2} q^{2\binom{j}{2}}(xyq)^j\\
	&=\sum_{i\ge 0}\sum_{j\ge 0}x^{i+j}y^jq^{i^2+j^2+i}\qbinom{i+1}{j}_{q^2}.
\end{align*}
Recall that
\begin{align*}
	S(x)&=\sum_{n\ge 0}s_n x^n\\
	\text{\tiny(by \eqref{eq:s-t})}&= \sum_{n\ge 0}\frac{t_n}{(q^2;q^2)_n} x^n.
\end{align*}
Therefore,
\begin{align*}
	S(x)=\sum_{i\ge 0}\sum_{j\ge 0}\frac{x^{i+j}y^jq^{i^2+j^2+i}}{(q^2;q^2)_{i+j}}\qbinom{i+1}{j}_{q^2},
\end{align*}
which confirms \eqref{eq:gf-S} by recalling \eqref{eq:S-gf-H1}.

\subsection{Proof of (\ref{eq:gf-S-E})}

The proof can be processed in the same way. We have
\begin{align*}
	1+\sum_{\lambda\in \cS_E} x^{\sharp(\lambda)} y^{\sharp_{1,2}(\lambda)}q^{|\lambda|} &= H_{1,1}(x)+ H_{3,1}(x) + H_{4,1}(x) + H_{5,1}(x)\\
	&=H_{1,1}(xq^{-2}).
\end{align*}
Letting
\begin{align*}
	S_E(x):=H_{1,1}(xq^{-2}),
\end{align*}
we also have
\begin{align*}
	q^2 S_E(x)-(1+q^2+xq^4+x^2yq^5)S_E(xq^2) + (1+xq^4)S_E(xq^4) = 0.
\end{align*}
If we further write
\begin{align*}
	S_E(x) := \sum_{n\ge 0}s_{E,n} x^n,
\end{align*}
then
\begin{align*}
	s_{E,0}&=1,\\
	s_{E,1}&=q^2+q^4+\cdots = \frac{q^2}{1-q^2}.
\end{align*}
Defining
\begin{align*}
	t_{E,n} := s_{E,n} (q^2;q^2)_n
\end{align*}
and writing
\begin{align*}
	T_E(x):=\sum_{n\ge 0}t_{E,n} x^n,
\end{align*}
we arrive at
\begin{align*}
	T_E(x)-xq^2(1+xyq)T_E(xq^2) = 1.
\end{align*}
It follows that
\begin{align*}
	T_E(x)-x^Nq^{N(N+1)}(-xyq;q^2)_N T_E(xq^{2N}) = \sum_{i=0}^{N-1} x^i q^{i(i+1)} (-xyq;q^2)_{i},
\end{align*}
so that
\begin{align*}
	T_E(x) &= \sum_{i\ge 0}x^i q^{i(i+1)}(-xyq;q^2)_{i}\\
	\text{\tiny(by \eqref{eq:q-binom})}&= \sum_{i\ge 0}\sum_{j\ge 0}x^{i+j}y^jq^{i^2+j^2+i}\qbinom{i}{j}_{q^2}.
\end{align*}
Thus,
\begin{align*}
	S_E(x)=\sum_{i\ge 0}\sum_{j\ge 0}\frac{x^{i+j}y^jq^{i^2+j^2+i}}{(q^2;q^2)_{i+j}}\qbinom{i}{j}_{q^2},
\end{align*}
as desired.

\section{Proof of Theorem \ref{th:gf-EB}}\label{sec:gf-EB}

For the proof of Theorem \ref{th:gf-EB}, we also note first that
\begin{align*}
	\sum_{\lambda\in \cB_O} x^{\sharp(\lambda)} y^{\sharp_{1,2}(\lambda)}q^{|\lambda|} &= \sum_{\lambda\in \cB} x^{\sharp(\lambda)} y^{\sharp_{1,2}(\lambda)}q^{|\lambda|} - \sum_{\lambda\in \cB_E} x^{\sharp(\lambda)} y^{\sharp_{1,2}(\lambda)}q^{|\lambda|}\\
	&= \sum_{i\ge 1}\sum_{j\ge 0}\frac{x^{i+j}y^jq^{i^2+j^2+i+2j}}{(q^2;q^2)_{i+j}}\left(\qbinom{i}{j}_{q^2}-\qbinom{i-1}{j}_{q^2}\right)\\
	\text{\tiny(by \eqref{eq:q-binom-1})}&=\sum_{i\ge 1}\sum_{j\ge 0}\frac{x^{i+j}y^jq^{i^2+j^2+3i}}{(q^2;q^2)_{i+j}}\qbinom{i-1}{j-1}_{q^2}.
\end{align*}
Thus, it is only necessary to establish \eqref{eq:gf-EB} and \eqref{eq:gf-EB-E}.

\subsection{Proof of (\ref{eq:gf-EB})}

We still deduce from Proposition \ref{prop:ptn-set-new} that
\begin{align*}
	1+\sum_{\lambda\in \cB} x^{\sharp(\lambda)} y^{\sharp_{1,2}(\lambda)}q^{|\lambda|} = H_{2}(x)+ H_{4}(x),
\end{align*}
so that by \eqref{eq:q-sys-1-2},
\begin{align}\label{eq:B-gf-H2}
	1+\sum_{\lambda\in \cB} x^{\sharp(\lambda)} y^{\sharp_{1,2}(\lambda)}q^{|\lambda|} = H_{2}(xq^{-2}).
\end{align}
Let us write
\begin{align*}
	B(x):=H_{2}(xq^{-2}).
\end{align*}

We then deduce from \eqref{eq:q-sys-2-2} that
\begin{align*}
	H_1(xq^2) = x^{-1}q^{-4}\big(H_2(x)-H_2(xq^2)\big).
\end{align*}
Substituting the above into \eqref{eq:q-sys-2-1} implies that
\begin{align}\label{eq:B-q-dif}
	q^2 B(x)-(1+q^2+xq^4+x^2yq^7)B(xq^2) + (1+xq^4)B(xq^4) = 0.
\end{align}

Writing
\begin{align*}
	B(x) := \sum_{n\ge 0}b_n x^n,
\end{align*}
then
\begin{align*}
	b_0&=1,\\
	b_1&=q^2+q^4+\cdots = \frac{q^2}{1-q^2},
\end{align*}
and for $n\ge 2$,
\begin{align*}
	(1-q^{2n})(1-q^{2n-2})b_n = q^{2n} (1-q^{2n-2}) b_{n-1} + yq^{2n+1} b_{n-2}.
\end{align*}

Let
\begin{align*}
	c_n:=b_n(q^2;q^2)_n
\end{align*}
and
\begin{align*}
	C(x):=\sum_{n\ge 0}c_n x^n.
\end{align*}
Then
\begin{align}
	C(x)-xq^2(1+xyq^3)C(xq^2)=1.
\end{align}
We find that
\begin{align}
	C(x)=\sum_{i\ge 0}x^iq^{i(i+1)}(-xyq^3;q^2)_i,
\end{align}
whose truncated version is
\begin{align*}
	C(x)-x^Nq^{N(N+1)}(-xyq^3;q^2)_N C(xq^{2N}) = \sum_{i=0}^{N-1} x^i q^{i(i+1)} (-xyq^3;q^2)_i,
\end{align*}
which can be easily shown by induction on $N$.

Finally, by \eqref{eq:q-binom},
\begin{align*}
	C(x)=\sum_{i\ge 0}\sum_{j\ge 0}x^{i+j}y^jq^{i^2+j^2+i+2j}\qbinom{i}{j}_{q^2},
\end{align*}
and therefore,
\begin{align*}
	B(x)=\sum_{i\ge 0}\sum_{j\ge 0}\frac{x^{i+j}y^jq^{i^2+j^2+i+2j}}{(q^2;q^2)_{i+j}}\qbinom{i}{j}_{q^2}.
\end{align*}

\subsection{Proof of (\ref{eq:gf-EB-E})}

We carry out similar arguments and get
\begin{align*}
	\sum_{\lambda\in \cB_E} x^{\sharp(\lambda)} y^{\sharp_{1,2}(\lambda)}q^{|\lambda|} &= H_{2,1}(x)+ H_{4,1}(x)\\
	&=H_{2,1}(xq^{-2}).
\end{align*}
Letting
\begin{align*}
	B_E(x):=H_{2,1}(xq^{-2}),
\end{align*}
then
\begin{align*}
	q^2 B_E(x)-(1+q^2+xq^4+x^2yq^7)B_E(xq^2) + (1+xq^4)B_E(xq^4) = 0.
\end{align*}
Now we write
\begin{align*}
	B_E(x) := \sum_{n\ge 0}b_{E,n} x^n,
\end{align*}
and note that
\begin{align*}
	b_{E,0}&=0,\\
	b_{E,1}&=q^2+q^4+\cdots = \frac{q^2}{1-q^2}.
\end{align*}
Let
\begin{align*}
	c_{E,n} := b_{E,n} (q^2;q^2)_n
\end{align*}
and
\begin{align*}
	C_E(x):=\sum_{n\ge 0}c_{E,n} x^n.
\end{align*}
Then
\begin{align*}
	C_E(x)-xq^2(1+xyq^3)C_E(xq^2)=xq^2.
\end{align*}
It follows that
\begin{align*}
	C_E(x)-x^Nq^{N(N+1)}(-xyq^3;q^2)_N C_E(xq^{2N}) = \sum_{i=1}^{N} x^i q^{i(i+1)} (-xyq^3;q^2)_{i-1},
\end{align*}
so that
\begin{align*}
	C_E(x) &= \sum_{i\ge 1}x^i q^{i(i+1)} (-xyq^3;q^2)_{i-1}\\
	\text{\tiny(by \eqref{eq:q-binom})}&= \sum_{i\ge 1}\sum_{j\ge 0} x^{i+j}y^jq^{i^2+j^2+i+2j}\qbinom{i-1}{j}_{q^2},
\end{align*}
which finally yields
\begin{align*}
	B_E(x)=\sum_{i\ge 1}\sum_{j\ge 0}\frac{x^{i+j}y^jq^{i^2+j^2+i+2j}}{(q^2;q^2)_{i+j}}\qbinom{i-1}{j}_{q^2}.
\end{align*}

\section{Conclusion}\label{sec:rmk}

The previous applications of linked partition ideas in the literature were usually restricted to partitions under conditions on the difference of neighboring parts. Our paper seems to be the first with parity conditions on adjacent parts considered. Such an analysis is made possible by subtly assigning different names to the empty partition in the linked partition ideal decomposition. More generally, we may carry out the same idea to cope with partitions under one or more conditions such as the prohibition or allowance of adjacent parts $\lambda_i$ and $\lambda_{i+1}$ with $\lambda_i\equiv a \pmod{m}$ and $\lambda_{i+1}\equiv a' \pmod{m'}$. It is expected that the advent of such refinements shall bring about more potential for the use of linked partition ideas in the investigation of generating functions for partitions.

\subsection*{Acknowledgements}

I am grateful to George Andrews for helpful discussions on an earlier version of this paper. This work was supported in part by a Killam Postdoctoral Fellowship from the Killam Trusts.

\bibliographystyle{amsplain}

\begin{thebibliography}{9}
	
	\bibitem{And1972}
	G. E. Andrews, Partition identities, \textit{Advances in Math.} \textbf{9} (1972), 10--51.
	
	\bibitem{And1974b}
	G. E. Andrews, A general theory of identities of the Rogers--Ramanujan type, \textit{Bull. Amer. Math. Soc.} \textbf{80} (1974), 1033--1052.
	
	\bibitem{And1975}
	G. E. Andrews, Problems and prospects for basic hypergeometric functions, In: \textit{Theory and application of special functions (Proc. Advanced Sem., Math. Res. Center, Univ. Wisconsin, Madison, Wis., 1975)}, 191--224, Math. Res. Center, Univ. Wisconsin, Publ. No. \textbf{35}, Academic Press, New York, 1975.
	
	\bibitem{And1976}
	G. E. Andrews, \textit{The theory of partitions}, Reprint of the 1976 original. Cambridge University Press, Cambridge, 1998.
	
	\bibitem{And2022}
	G. E. Andrews, Separable integer partition classes, \textit{Trans. Amer. Math. Soc. Ser. B} \textbf{9} (2022), 619--647.
	
	\bibitem{ACL2021}
	G. E. Andrews, S. Chern, and Z. Li, Linked partition ideals and the Alladi--Schur theorem, \textit{J. Combin. Theory Ser. A} \textbf{189} (2022), Paper No. 105614, 19 pp.
	
	\bibitem{ADR2020}
	G. E. Andrews, V. Dragovi{\'c}, and M. Radnovi{\'c}, Combinatorics of periodic ellipsoidal billiards, \textit{Ramanujan J.}, accepted. \doi{10.1007/s11139-020-00346-y}.
	
	\bibitem{CL2018}
	S. Chern and Z. Li, Linked partition ideals and Kanade--Russell conjectures, \textit{Discrete Math.} \textbf{343} (2020), no. 7, 111876, 24 pp.
	
	\bibitem{DR2015}
	V. Dragovi{\'c} and M. Radnovi{\'c},
	Periodic ellipsoidal billiard trajectories and extremal polynomials, \textit{Comm. Math. Phys.} \textbf{372} (2019), no. 1, 183--211.

\end{thebibliography}

\end{document}